\documentclass[a4paper,12pt]{article}

\usepackage{fullpage}
\usepackage{amsthm,amsmath, amssymb}
\usepackage{palatino,comment}
\usepackage[usenames, dvipsnames]{color}
\usepackage[active]{srcltx}   
\input{xy}
\xyoption{all}
\xyoption{matrix}

\usepackage{hyperref}
\usepackage{calrsfs}
\DeclareMathAlphabet{\pazocal}{OMS}{zplm}{m}{n}

\newtheorem{theorem}{Theorem}[section]
\newtheorem{teo}[theorem]{Theorem}

\newtheorem{lem}[theorem]{Lemma}

\newtheorem{prop}[theorem]{Proposition}

\newtheorem{coro}[theorem]{Corollary}

\theoremstyle{definition}
\newtheorem{definition}[theorem]{Definition}

\newtheorem{defi}[theorem]{Definition}

\newtheorem{ex}[theorem]{Example}

\newtheorem{rem}[theorem]{Remark}

\DeclareMathOperator{\Der}{Der}

\DeclareMathOperator{\Hom}{Hom}

\DeclareMathOperator{\Ker}{Ker}

\def\Im{\mathrm{Im\ }}

\DeclareMathOperator{\id}{id}

\DeclareMathOperator{\End}{End}

\newcommand{\ot}{\otimes}

\def\X{\mathfrak{X}}
\def\div{\mathrm{div}}

\def\aff{\mathfrak{aff}}
\def\Vol{\mathrm{Vol}}

\newcommand{\wt}{\widetilde}
\newcommand{\wh}{\widehat}

\def\E{\mathcal E}

\def\D{\mathcal D}

\def\Ext{\mathrm{Ext}}
\def\Hom{\mathrm{Hom}}

\def\sl{\mathfrak{sl}}

\def\g{\mathfrak g}

\def\h{\mathfrak h}

\def\C{\mathbb C}
\def\R{\mathbb R}
\def\Z{\mathbb Z}

\def\ad{\mathrm{ad}}
\def\om{\omega}
\def\To{\Rightarrow}

\def\Om{\Omega}

\def\beq{\begin{equation}}
\def\eeq{\end{equation}}

\def\dpi{\partial_\Pi}
\def\dce{\partial_{CE}}

\def\pf{\begin{proof}}
\def\epf{\end{proof}}

\def\dd{\partial_\delta}

\def\dce{\partial_{CE}}

\author{Marco A. Farinati
\thanks{Dpto de Matem\'atica FCEyN UBA - IMAS (Conicet). 
e-mail: mfarinat@dm.uba.ar.
} \ and
A. Patricia Jancsa
\thanks{Dpto de Matem\'atica FCEyN UBA. 
e-mail: pjancsa@dm.uba.ar.
Both authors are
partially supported by 
UBACyT and
PICT 2018-00858 ``Aspectos algebraicos y anal\'iticos de grupos
 cu\'anticos''.}
 }

\date{}

\title{BV - Algebras with Applications to Lie Bialgebras and 
Poisson Cohomology}

\begin{document}
\maketitle

\begin{abstract}
By using algebraic tools from differential Gerstenhaber 
algebras and Batalin - Vilkovisky algebras,
we provide  a new perspective  on the modular class
in Poisson geometry and  the intrinsic biderivation 
of a Lie  bialgebra. Furthermore,  applications in Lie bialgebra cohomology and Poisson cohomology are given.
\end{abstract}

\maketitle
 
\section*{Introduction}

 Gerstenhaber algebras derive their name after
  Gerstenhaber's
  work on  algebraic structure for Hochshild cohomology.
   However, examples of Gerstenhaber algebras were known
 for a long time: the full antisymmetric multivector fields on a manifold form a Gerstenhaber algebra with the Schouten - Nijenhuis bracket as ``Gerstenhaber'' bracket. Also, the
 complex computing Lie algebra homology is another example of Gerstenhaber algebra. In this paper we study Gerstenhaber algebras of this type, and the relation with being a BV - algebra.
 Again, the Chevalley-Eilenberg differential of the Chevalley-Eilenberg complex is an example of BV -  structure. In this 
 algebraic context, it is not new to consider also
 Poisson structures. However, what is new is the connection between the modular class of a Poisson manifold and the BV - formalism.
 
  The content of this work is as follows.
  In {\em Section 1} we recall
 the definitions of Gerstenhaber and BV-algebras, along with 
 the examples wich are relevant to this work.
 Also, we present a characterization of BV  - structures on Lie
 algebroids (or more generally, Lie - Rinehart algebras)
 in terms of divergence operators (see Proposition \ref{propdiv}).
 
 In {\em Section 2} we recall the  presentation of Poisson 
 structures using the algebraic formalism of Gerstenhaber algebras.
 We prove that a differential of a Gerstenhaber structure,
 that comes from a BV - algebra, may not automatically be compatible with the BV - structure; that is, the differential of a d.g. Gerstenhaber algebra does not necessarily anticommute with a BV - differential that generates the Gerstenhaber bracket. Next, we prove our main result, Theorem \ref{main},
 which proves that the aforementioned anticommutator is,
 in fact,  given by a derivation. Furthermore, we provide two significant applications: one for Poisson structures and another concerning Lie bialgebras.
 
 In {\em Section 3}, we show that, in the Poisson case,
 this derivation corresponds to the inner derivation 
  determined
  by the  modular class
 of the Poisson structure (see Theorem \ref{bider}). 
 Besides,
 after a discussion on unimodularity,
   we provide a new
 proof of the fact that Poisson cohomology is a BV - algebra when 
 the Poisson structure is unimodular
 (see Theorem \ref{unimod}).

 The application to Lie bialgebras is given in {\em Section 4}. We first show that this derivation is
  determined by the intrinsic (bi)derivation of the Lie bialgebra
  (see Theorem \ref{teobialg}).
  An explicit application is provided, exhibing how the previous
  algebraic machinery simplifies a cohomological calculus in Lie bialgebra cohomology (see Corollary \ref{corodiag}).

  Finally, in {\em Section 5}, we present the classical example of the Poisson
  manifold $\R^2$ with Poisson bracket $(x ^2+y ^2)\partial_x\wedge\partial_y$, viewed
  under the light of our results. We give an
  explicit calculation
  of its Poisson cohomology using a smaller complex, that
  gives  specific generators (see Theorem
  \ref{R2}).

\section{Batalin-Vilkovisky and Gerstenhaber Algebras}

\subsection{Gerstenhaber algebras}

We begin by recalling the definition of
Gerstenhaber algebras:

\begin{defi}
A Gerstenhaber algebra is a graded-commutative algebra
$A$ together with a bracket $[-,-]:A\times A\to A$
satisfying
\begin{itemize}
\item $A[1]$ is a super-Lie algebra, where $A[1]$ means
$A$ shifted by one as graded object.
\item $[a,bc]=[a,b]c+(-1)^{(|a|-1)|b|}b[a,c]$ where
$a$ and $b$ are homogeneous elements and
$|a|$ and $|b|$ means the (original) grading of $a$ and $b$
respectively.
\end{itemize}
\end{defi}

\subsection{Important examples of Gerstenhaber algebras}

\begin{enumerate}
\item For any associative algebra $A$ over a field $k$,
 its Hochschild
cohomology 
\[
HH ^\bullet(A)=H ^\bullet(A,A)
=\Ext_ {A^e}(A,A)
\]
is a Gerstenhaber algebra with cup product
(in particular the cup product is super-commutative)
and the bracket given by the antisimmetrization of the
Pre-Lie structure found by Gerstenhaber in the 
standar complex
\cite{G}.
 If $A$ is commutative
 regular, this structure recovers a well-known
geometric case that is the following example:

\item If $M$ is a $C^{\infty}$-manifold, then
\[
\X ^\bullet(M):=\Lambda_{C ^\infty(M) }^\bullet\X(M)
\]
 the algebra of
full-antisymmetric multivector fields is a Gerstenhaber algebra
 with exterior super commutative product 
and the Schouten-Nijenhuis bracket.
If $M=G$ is a Lie group, the $G$-invariant part is a
subobject that gives the following purely alegbraic example:

\item If $\g$ is a Lie algebra over a field $k$, then the exterior
 algebra
\[
\Lambda_k ^\bullet \g
\]
is a  Gerstenhaber algebra. The wedge product is
the supercommutative product, and the
Lie bracket is defined as the unique extension
  (by derivation of the wedge product) 
of the Lie bracket of $\g$. Another algebraic example that
generalizes both this example and the previous one is the following:

\item Assume a Lie-Rinehart algebra $(A,L)$ is given.
That is, a commutative algebra $A$ and a Lie algebra $L$
that is also an $A$-module, together with a Lie
 algebra map
(the so-called anchor map)
\[
\alpha:L\to \Der(A)
\]
\[
X\mapsto X_\alpha\]
satisfying
\[
[X,fY]_L=f[X,Y]_L+X_\alpha(f): f\in A, \ X,Y\in L
\]
If $A=C ^\infty(M)$ and $L$ is  proyective $A$-module then it
is also called a {\em Lie algebroid}. The first example
is when $A=C ^\infty(M)$ and $L=\Der(A)$ with identity
 as the anchor map. Another example is 
a ''Lie algebroid over a point'', where
 $A=k$ and  $L=\g$,  a $k$-Lie algebra with and zero anchor map.

In this situation, 
\[
\Lambda ^\bullet_AL
\]
is a Gerstenhaber algebra with the unique Gerstenhaber bracket satisfying
\[
\forall X,Y\in L: [X,Y]=[X,Y]_L
\]
\[
\forall f\in A, X\in L:
[X,f]=X_\alpha(f)
\]

\item Another example of Lie algebroid is  $\Om^1(M)$ 
where $M$ is a Poisson manifold, with anchor map
\[
\alpha:\Om^1(M)\to \X(M)
\]
\[fdg\mapsto f\{g,-\}\]
and Poisson bracket
determined by
\[
[df,dg]=d\{f,g\}
\]
\end{enumerate}

\begin{rem}
The tensor product of Gerstenhaber algebras is naturally a Gerstenhaber algebra. This is clear for the super-commutative
product using the sign rule
\[
(a  b)(a' b')=(-1)^{|b|\ |a'|}(aa')(  bb')
\]
but one can also define the Lie bracket
using the anti symmetry and the rule
\[
[x,ab]=[x,a]b+(-1)^{(|x|-1)|a|}a[x,b]
\]
\end{rem}

\subsection{Batalin-Vilkovisky algebras}

Batalin-Vilkovisky structures arose first in physics as a
 generalization of the BRST formalism for quantizing gauge
theories. However, from a mathematical point of view,
this notion is evidently related to Gerstenhaber algebras
and can be understood and applied to mathematical objects
 without any a-priori connection to  physical gauge theories.

\begin{defi} A Batalin-Vilkovisky algebra, or BV algebra 
for short, 
 is a $\Z$-graded,
supercommutative associative
algebra  together with an operator $\Delta$ of degree -1
satisfying
\begin{enumerate}
\item $\Delta ^2=0$
\item 
$
\Delta(abc)-\Delta(ab)c
+\Delta(a)bc-(-1)^{|a|}a\Delta(bc)
-(-1) ^{(|a|+1)|b|}b\Delta(ac)+$
\[
+(-1)^{|a|}a\Delta(b)c+(-1)^{|a|+|b|}ab\Delta(c)
-\Delta(1)abc=0\]
\end{enumerate}
\end{defi}

Sometimes it is required that $\Delta(1)=0$, but if
 the grading of the algebra is non-negative this is trivially
 satisfied.
 
A well-known fact about BV algebras 
is that the bracket defined
as the failure of $\Delta$ to be a derivation (or
more precisely, a differential operator of degree $\leq$1):
\[
[a,b]_\Delta:=(-1)^ {|a|}\Delta(ab)-(-1)^{|a|}\Delta(a)b-a\Delta(b)+a\Delta(1)b\]
This bracket turns $A$ into a Gerstenhaber algebra,
so BV-algebras are, in particular, Gerstenhaber algebras.
Gerstenhaber algebras admiting a BV operator $\Delta$
that recovers the bracket are sometimes called {\em exact}
Gerstenhaber algebras, and $\Delta$ is said to generate
the bracket.

\begin{rem}
If  $(A,\Delta_A)$ and $(B,\Delta_B)$ are
BV-\'algebras with $\Delta(1)=0$ then
$A\ot B$ is also a BV-algebra. The BV operator 
in $A\ot B$ is defined in the usual way:
\[
\Delta(a\ot b):=\Delta(a)\ot b+(-1) ^{|a|}a\ot \Delta(b)
\]
It clearly satisfies $\Delta^2=0$. but it is a remarkable fact that
it also satisfies the second  condition.
\end{rem}

\subsection{Our examples of interest of BV algebras}

 $\Lambda ^\bullet\g$ is in fact a BV-algebra with $\Delta=\dce$ the Chevalley-Eilenberg differential
\[
\Delta(x_1\wedge\cdots\wedge x_n)=\sum_{i<j}(-1) ^{i+j}
[x_i,x_j]\wedge x_1\wedge\cdots \wh{x_i}\cdots\wh{x_j}\cdots\wedge x_n
\] 
For instance, if $x_1,x_2\in \g\subset\Lambda ^\bullet\g$ then
\[
[x_1,x_2]_{\Delta}=(-1) ^{1}\Delta(x_1,x_2)=-(-1)^{1+2}[x_1,x_2]_\g=[x_1,x_2]_\g
\]
Since a Gerstenhaber bracket is determined by its values in a set of generators as associative algebra, the above checking 
for $x_1,x_2\in\g$ shows that the Gerstenhaber bracket 
in $\Lambda_k^\bullet\g$ agrees with the one generated by
 $\dce$.

Another classical example is
 $\X ^\bullet (M)$, with $M$ an
 orientable manifold. Pick $\om_0$ a volume form, this defines a divergence operator
 $\div:\X(M)\to C ^\infty(M)$ in terms of the Lie derivative of the volume form: for any vector field 
$X$, the Lie derivative of the volume form is necesarily a multiple
of the volume form, so one can define 
the divergence of a vector field by the following formula
\[
L_X(\om_0)=d~ \iota_X(\om_0)=:\div(X)\om_0
\]
Maybe it should be denoted $\div_0$ 
since it dependes on the choice of the volume form $\om_0$. 
\begin{rem}
\begin{itemize}
\item If $a$ is a nowhere vanishing function and 
$\om_1=a\om_0$, then
\[
L_X(a\om_0)
=X(a)\om_0+a\div(X)\om_0
\]
\[
=\frac1aX(a)a\om_0+\div(X)a\om_0
=\frac1aX(a)\om_1+\div(X)\om_1\]
\[
=X(\ln|a|)\om_1+\div(X)\om_1
\]
So, the divergence using another volume form agrees up to
the exact differential $d(\ln|a|)$.
\item
We remark that the orientability condition is not really 
necesary,
as we all know the local formula for the divergence in a 
Riemannian manifold.
\end{itemize}
\end{rem}

With this operator one can define a BV-operator on multivector
 fields. Additionally, one can make the same
 contruction for Lie - Rinehart algebras:

\begin{defi}\label{defdiv}
Let $(L,A)$ be a Lie - Rinehart algebra with anchor map $\alpha:L\to \Der(A)$. 
For simplicity, if $X\in L$ and $a\in A$,
 denote $X(a)$ instead of $X_\alpha(a)$
or $\big(\alpha(X)\big)(a)$.
A {\bf divergence}
 operator
is an additive map
\[
\div:L\to A\]
satisfying the following two conditions
\[
\div[X,Y]=X(\div Y)-Y(\div X)
\] 
\[
\div(aX)=a\div X+X(a)
\]
\end{defi}
In \cite{Hue} an equivalent notion is described in terms 
of (algebraic) right connections, and it is
 proven that if $L$ is proyective
of constant rank $n$, then one can give a right connection 
in $A$ from a one in 
$\Lambda_A ^nL$ \cite[Theorem 3]{Hue}. 
In the geometric setting this was done before by Ping Xu
(see \cite{X}), in our notation, if $X\in L$ and 
$\om\in \Lambda ^n_AL$ (assuming $L$ is projective of
 constant rank $n$ so that 
 $\End_A(\Lambda ^n_AL)= A\cdot\id$,
 $\nabla$ is a flat $A$-connection, and denoting
  $L_X=[X,-]$, 
 then Xu's formula 
 \cite[Proposition 3.11]{X} is given by
\[
L_X(\om)-\nabla_X\om=(\div_\nabla)\om
\]
Heubschmann's notation is $\div(X)=1\circ X$.
In terms of an abstract divergence operator we have:

\begin{prop}\cite[Theorem 1]{Hue}, 
\cite[Proposition 3.1]{X} If $(A,L)$ is a Lie Rinehart pair and
 $\div:L\to A$ is divergence operator, then the following is a
BV-operator generating the Gerstenhaber bracket on
$\Lambda_A ^\bullet L$:
\[
\Delta(X_1\wedge\cdots\wedge X_n)=\sum_{i<j}(-1) ^{i+j}
[X_i,X_j]\wedge X_1\wedge\cdots \wh{X_i}\cdots\wh{X_j}\cdots\wedge X_n
\]
\[
+\sum_{i=1}^ {n}
(-1)^{i}\div(X_i)
 X_1\wedge\cdots \wh{X_i}\cdots\wedge X_n
\] 
\end{prop}

We remark that the following reciprocal statement is also true:

\begin{prop}\label{propdiv}
For a Lie Rinehart algebra $(A,L)$ the following data
 are equivalent:
\begin{enumerate}
\item A BV-differential generating the Gerstenhaber bracket
on $\Lambda_A ^\bullet L$.
\item A divergence operator $\div:L\to A$ satisfying the
conditions of Definition \ref{defdiv}.
\end{enumerate}
\end{prop}

\begin{proof}
$2.\To 1.$ This are Huebschmann's computations. 

\

$1\To 2$. If $\Delta$ is a BV operator on $\Lambda_A ^\bullet L$
\[
\cdots\to \Lambda ^2_AL\overset{\Delta}{\longrightarrow}
\Lambda ^2_AL\overset{\Delta}{\longrightarrow}
L\overset{\Delta}{\longrightarrow}A\to 0
\]
 then $\Delta(A)=0$ just by degree reasons, and 
$\Delta(L)\subseteq A$. Define
\[
\div:=-\Delta|_L:L\to A
\]
Using the fact  that $\Delta$ generates the Gerstenhaber bracket,
we have in particular that, for $a\in A$ and $X\in L$:
\[
X(a)=
[X,a]
=(-1)^ {|X|}\Delta(Xa)-(-1)^{|X|}\Delta(X)a-X\Delta(a)\]
\[
=+\div(aX)-a\div(X)\]
or equivalently, $\div(aX)=a\div(X)+X(a)$.

Also, for $X,Y$ in $L$, the general formula
\[
[X,Y]_\Delta=(-1)^ {|X|}\Delta(X\wedge Y)-(-1)^{|X|}\Delta(X)\wedge Y-X\wedge \Delta(Y)\]
gives
\[
[X,Y]=-\Delta(X\wedge Y)-\div(X) Y+\div(Y)X\]
Applying $\Delta$ we get
\[
-\div([X,Y])=\Delta([X,Y])
=-\Delta ^2(X\wedge Y)-\Delta(\div(X) Y)+\Delta(\div(Y)X)\]
since $\Delta^2=0$ and $\Delta|_L=-\div$,
\[
=\div(\div(X) Y)-\div(\div(Y)X)\]
\[
=\div(X) \div(Y)
+Y(\div X)
-\div(Y)\div(X)
-X(\div Y)
\]
\[
=
+Y(\div X)
-X(\div Y)
\]
That is, $\div[X,Y]=X(\div Y)-Y\div(X)$.
\end{proof}

\begin{rem}
One can say that the formulas
$\div[X,Y]=X\div(Y)-Y\div(X)$ and
$\div(aX)=a\div(X)+X(a)$ are necessary for $\Delta ^2$
to be zero, but we have just seen that they are suficient.
\end{rem}

\begin{ex}
If $M$ is a Poisson Manifold, $A=C ^\infty(M)$ and
$L=\Om^1(M)$, then one can define the map
\[
\Om ^1(M)\to C^\infty(M)\]
\[
fdg\mapsto \{f,g\}
\]
One can easily check that it satisfies the conditions of Definition
\ref{defdiv}, so it is a divergence operator. The corresponding
BV operator agrees  with Brilinsky's differential $d_{\{,\}}$
\cite{Br}. 
\end{ex}

\begin{rem}
Usually, BV algebras are though of as alternatives to algebras
 where some integration is possible, or are related to 
 the existence
 of a volume form or a dualizing bimodule 
(or complex) such as $\Om ^{\dim M}(M)$ or 
$\Ext_{A^e} ^\bullet(A,A ^e)$.
 We also  mention that  
T. Tradler \cite{TT} shows a relation between being
a Frobenius algebra and having a BV structure on Hochschild
 cohomology. Huesbschann shows how to produce a 
right $L$-connection on $A$ from a connection on 
$\Lambda ^n_AL$.
We mention in this direction 
the notion of a calculus with duality \cite{La} and
the compatibilities of
 BV-structures found in the context of 
Koszul Calabi-Yau algebras \cite{CYZ}.
Also, at very basic level,  
 the divergence operator is ussually presented 
as measuring whether  a vector field preserves  the volume
or not.
However, we would like to emphasis the simple and elementary
abstract conditions of Definition \ref{defdiv} that 
produce BV-structures  in situations where 
no volume form, nor dualizing complex is given, or it is hard to 
fomalize in a satisfactory way, but the 
divergence operator may be obvious, as the following
algebraic example shows.
\end{rem}

\begin{ex}
Let $A=k[\{x_i\}:i\in I]$ 
be a polynomial ring in infinitely many variables and
$L=\bigoplus_{i\in I}A\partial_i$ be 
the obvious sub-Lie algebra
of $\Der_k(A)$. The inclusion $L\to \Der(A)$ defines
the anchor map
of the Lie Rinehart
pair $(A,L)$. It is clear that for $I$ infinite,  there is no
volume form, nor dualizing complex. However, 
for $X=\sum\limits_{i\in I} p_i\partial_i\in\bigoplus\limits_i A\partial_i$,
the operator defined by
\[
\div(X)=\div\big(\sum_i p_i\partial_i\big):=
\sum_i \frac{\partial p_i}{\partial x_i}
\]
satisfies the conditions of being a divergence operator. Hence,
we have a BV operator $\Delta$ generating the Gerstenhaber algebra structure of $\Lambda_A^\bullet L$.
\end{ex}

\begin{rem}\label{divtilde}
If $\div:L\to A$ is a divergence operator and  $a_0\in A$ then
\[
\wt \div(X):=\div(X)+X(a_0)
\]
is also a divergence operator.
The corresponding BV operator is
\[
\wt\Delta=\Delta-[a,-]
\]
\end{rem}
\begin{proof}
Straightforward.
\end{proof}

\begin{rem}
Keeping the above notation, one may check that the cohomology with respect to $\Delta$ is
isomorphic to the cohomology with respect to $\wt\Delta$.
However,
 in general,  the same Gerstenhaber algebra 
may admit several BV-structures that are non-equivalent
in the sense that, for example,
their cohomology is different.  The minimal example
is the following:
\end{rem}
 
 \begin{ex}
 Consider $A:=k[x]/(x ^2)=k\oplus kx$ with the grading $|x|=1$.
One can trivially check that the  trivial Gerstenhaber bracket
is the only possible one. However, for the BV-structure, for any
fixed $\lambda\in k$ one may define
\[
\Delta_\lambda(x):=\lambda,\ \ \Delta_\lambda(1)=0
 \]
One can easily see that $\Delta_\lambda$ satisfies 
the BV-conditions and generate the (only possible)
 Gerstenhaber structure.
The cohomology with respect to $\Delta_\lambda$ 
clearly depends on $\lambda$:
 \[
 \begin{array}{ccccccc}
\lambda\neq 0& \To &   
\Delta_\lambda\colon kx\overset{\cong}{\longrightarrow   }k
&\To& 
H_\bullet(k[x]/(x^2),\Delta_\lambda)&=&0\\
\lambda= 0& \To &   
\Delta_\lambda\colon kx\overset{0}{\longrightarrow   }k
&\To&
 H_\bullet(k[x]/(x^2),\Delta_0)&=&k[x]/(x^2)
\end{array}
\]
\end{ex}

\section{(Maybe non) differential BV algebras
and the main result}

\subsection{Poisson structures}

A Poisson structure on a manifold $M$ is a Lie bracket
$\{-,-\}$ on the algebra of smooth functions that is a
 derivation on each variable. This data can be given 
by a 2-vector
\[
\Pi\in\X^ 2(M)
\]
fully determined by the formula
\[
\Pi(df,dg)=\{f,g\}
\]
It is a well-known fact that $\{-,-\}$ verifies Jacobi
 if and only if
the 2-vector $\Pi$ satisfies
\[
[\Pi,\Pi]_{SN}=0\in\X ^3(M)
\]
where $[-,-]_{SN}$ is the Schouten-Nijenhuis bracket.
In \cite{L} Lichnerowich
introduced the Poisson cohomology by considering the complex
\[
(\X ^\bullet(M),\dpi)
\]
where $\dpi$ is given by
\[
\dpi(X_1\wedge\cdots\wedge X_k)=
[\Pi,X_1\wedge\cdots\wedge X_k]
\]
In general, if $\Lambda ^\bullet$
 is a Gerstenhaber algebra
 (e.g. $\Lambda_A ^\bullet L$ for some Lie-Rinehart 
 pair) and $\Pi\in\Lambda ^2$ 
is an element satisfying 
\[
[\Pi,\Pi]=0\]
 then the (super)Jacobi identity for the Gerstenhaber bracket
 gives
\[
[\Pi,[\Pi,-]]=2[\Pi,-]=0
\]
and one may consider the complex
\[
(\Lambda ^\bullet,\dpi=[\Pi,-])
\]
It is clear that $\dpi$ is a (super)derivation for both 
 the associative product and the Gerstenhaber bracket,
so the cohomology with 
respect to $\dpi$ is also a Gerstenhaber algebra. 
Not every derivation of a Gerstenhaber algebra is necesarily 
of this type, so we recall a definition:

\subsection{Lie Bialgebroids and strong differential Gerstenhaber algebras}

\begin{defi}\label{Gdif}
If $\Lambda ^\bullet$ is a Gerstenhaber algebra, a 
linear homogeneous operator  
$\partial:\Lambda^\bullet
\to \Lambda ^{\bullet+1}$ is called a
{\bf strong differential}
 for the 
Gerstenhaber structure if $\partial^2=0$ and
it is a (super)derivation with respect to both the associative 
product and the bracket:
\[
\partial(ab)=\partial(a)b+(-1)^{|a|}a\partial(b)
\]
\[
\partial[a,b]=[\partial a,b]+(-1)^{|a|-1}[a,\partial b]
\]
for all $a,b$ homogeneous elements in $\Lambda ^\bullet$.
\end{defi}

Recall the definition of a Lie bialgebra:
\begin{defi}
A Lie bialgebra over $k$ is a $k$-Lie algebra $\g$ together with a map $\delta:\g\to\Lambda ^2\g$ called the co-bracket satisfying
the following conditions:
\begin{itemize}
\item $\delta ^*:\Lambda ^2\g ^*\to \g^*$ is a Lie bracket (co-Jacobi condition).

\item $\delta[x,y]=\ad_x(\delta y)-\ad_y(\delta x)$
(1-cocycle condition).

\end{itemize}
\end{defi}

\begin{rem}
A Lie algebroid is  Lie Rinehart pair $(A,L)$ 
of the form $A=C ^\infty(M)$ and $L$ being 
$A$-projective and finitely generated. 
A Lie {\em bi}algebroid consists of a Lie algebroid 
$(A,L)$ together with a Lie algebroid structure on
 $L ^*:=\Hom_A(L,A)$ and some 
compatibility condition that we shall not make explicit since
(see \cite{KS}) all this 
 is equivalent 
to giving a strong differential
 $\partial:\Lambda^\bullet _AL\to \Lambda ^{\bullet+1}_AL$
as in Definition \ref{Gdif}.
Notice that a Lie {\em bi}algebroid over a point
is the  same as a Lie {\em bi}algebra. More precisely, if $A=k$
is a field and  $L=\g$ is a Lie algebra over $k$ 
(viewed as a Lie Rinehart pair with zero
anchor map), then a differential 
$\partial:
\Lambda ^\bullet\g\to \Lambda ^{\bullet+1}\g$ 
is fully determined by its restriction 
\[
\delta:=\partial|_\g:\g\to\Lambda ^2\g
\]
 The square zero condition for $\partial$ is equivalent to 
 co-Jacobi condition
 for $\delta$, and being a derivation with respect to the
  Gerstenhaber bracket is equivalent to the 1-cocycle condition.
\end{rem}

\subsection{Main result}

In the case where
a Gerstenhaber algebra 
$\Lambda^\bullet$ 
has a strong differential 
$\partial:\Lambda ^\bullet\to\Lambda^{\bullet+1}$
(e.g.
$\partial=\dpi=[\Pi,-]$ in the Poisson case) and, in addition,
 the Gerstenhaber
 structure arises from a BV operator $\Delta$, one may
  wonder about
the compatibility between $\partial$ and $\Delta$.
In general $\Delta$ does not necesarily to anti-commute with
$\partial$, however, the anticommutator has the 
following general property:

\begin{teo}\label{main} Let $\Lambda ^\bullet$
be a Gerstenhaber algebra admiting a 
BV-operator $\Delta$ that generates its bracket. 
We  assume $\Delta(1)=0$.
If
$\partial$ is a strong differential
for the Gerstenhaber structure i.e.,
$\partial:\Lambda^\bullet\to
\Lambda ^{\bullet+1}$ is a square-zero
derivation of the product and the bracket,
 then
the (homogeneous of degree zero) map
\[
\partial\circ\Delta+\Delta\circ\partial
\]
is a  derivation of $\Lambda ^\bullet$
with respect to the product (and the bracket).
\end{teo}
\begin{proof}
The fact that it is a derivation with respect to the bracket is
 clear, because both $\partial$ and $\Delta$ are so, and a
  (super)commutator
of derivations is a derivation. The key point in this result
 is that $\partial$
is a derivation with respect to the associative product but 
$\Delta$
is not, it is a second-order differential operator, so, 
for free,
this commutator will be just  a second-order differential
 operator.
We well see that it is actually a derivation.

Let us consider the space of linear endomorphisms
\[
\E:=\End(\Lambda ^\bullet\!\ot\! \Lambda ^\bullet
~\oplus~\Lambda ^\bullet)
\]
 Since
$\Lambda ^\bullet\!\ot\! \Lambda ^\bullet
~\oplus~\Lambda ^\bullet$ is a super vector space,
 $\E$ is naturally
a super Lie algebra with super commutator, let us denote it by $[[-,-]]$.

Recall that the failure of  $\Delta$ being  a derivation
is precisely  the Gerstenhaber bracket.
Let us extend $\Delta$ to an operator defined
on  $\Lambda ^\bullet\ot \Lambda ^\bullet$
in the usual way (and by abuse of notation,
denote it by the same name)
\[
\Delta(u\ot v):=
\Delta(u)\ot v+(-1) ^{|u|}u\ot \Delta(v)
\]
where $u, v\in A$, and $u$ is homogeneous.
In this way we consider $\Delta\in\E$. Also, consider the multiplication map
\[
m:\Lambda ^\bullet\!\ot\! \Lambda ^\bullet\to \Lambda ^\bullet
\]
\[
a\ot b\mapsto ab
\]
as an element in $\E$ by declaring
\[
m(a\ot b,c):=(0,ab)\in\Lambda ^\bullet\!\ot\! \Lambda ^\bullet\oplus \Lambda ^\bullet\]
and similarly we consider map $G$ induced by
the Gerstenhaber bracket  defined by
\[
G(a\ot b,c):=(0,(-1) ^{|a|}[a,b])
\]
Now, the BV-formula 
\[
(-1)^{|a|}[a,b]=\Delta(ab)-\Delta(a)b-
(-1)^{|a|}a\Delta(b)
\]
reads
\[
[[\Delta,m]]= G
\]
It is worth noting that stating that a degree zero map
 $D:A\to A$
is a derivation (with respect to the  product)
is precisely the same as asserting  that 
the induced element (denoted by the same letter) in $\E$,
 defined by
\[
D(u\ot v,w):=(D(u)\ot v+u\ot D(v), D(w))
\]
satisfies
\[
[[D,m]]=0
\]
Now, let us compute the bracket and apply the Jacobi identity:
\[
[[~[[\Delta,\partial]],m]]=
\pm [[~[[\Delta,m]],\partial]]
\pm
[[~[[\partial,m]],\Delta]]
\]
Using that $[[\Delta,m]]= G$ 
and that $\partial$ is a derivation with respect to the 
product, we have $[[\partial,m]]=0$, which leads to
\[
[[~[[\Delta,\partial]],m]]=
\pm [[G,\partial]]
\pm
[[ 0,\Delta]]
\]
But also $\partial$ is a derivation with respect to the Gerstenhaber bracket,
we get $[[G,\partial]]=0$, which concludes the proof.

For the sake of completeness, we will make explicit the signs
in this final statement. Notice $G$ is a degree
-1 map, while $\partial$ is a degree 1 map. Hence,
\[
[[G,\partial]]=G\circ \partial+\partial\circ G\]
Evaluating this expression on elements, we obtain
\[
\big(G\circ \partial+\partial\circ G\big)(a\ot b,c)
=G\big(\partial a\ot b+(-1) ^{|a|}a\ot\partial b,\partial c\big)
+\partial\big(0,(-1) ^{|a|}[a,b]\big)
\]
\[
=\big(0,(-1)^{|\partial a|}[\partial a, b]\big)
+\big(0,[a,\partial b]\big)
+\big(0,(-1) ^{|a|}\partial[a,b]\big)
\]
Thus, $[[G,\partial]]=0$ if and only if
\[
(-1)^{|\partial a|}[\partial a, b] +[a,\partial b]+
(-1) ^{|a|}\partial[a,b]=0\]
or equivalently, using $|\partial a|=|a|+1$,
\[
\partial[a,b]=
+[\partial a, b] +(-1) ^{|a|-1}[a,\partial b]
\]
\end{proof}

\section{Applications to Poisson Cohomology: the Modular Class}

Our first main application is to show that, in the Poisson case,
 the derivation of
Theorem \ref{main} is given
by the so-called modular class.
To begin, let us recall the notion of the modular class of a 
Poisson manifold. Readers familiar
 with this concept may go directly to Theorem \ref{bider}.

Given $f_0\in C^{\infty}(M)$ we have the associated Hamiltonian vector field $X_{f_0}$
given by
\[
X_{f_0}:=\{f_0,-\}:C^{\infty}(M)\to C^{\infty}(M)
\]
\[\hskip 2cm
g\mapsto \{f_0,g\}
\]
Choose a divergence operator and consider the function
$\div(X_{f_0})\in C^\infty(M)$. If we consider $f$ as
 a variable then the assignemnt
\[
C^\infty(M)\to C^\infty(M)
\]
\[
f\mapsto \div(X_f)
\]
satisfies
\[
fg\mapsto 
\div(\{fg,-\})
=
\div(f\{g,-\})+
\div(g\{f,-\})
\]
\[=
f\div(\{g,-\})+\{g,f\}+
g\div(\{f,-\})+\{f,g\}
\]
\[
=
f\div(X_g)
+g\div(X_fx)
\]
Thus, the asigment $f\mapsto \div(X_f)$ is a derivation,
 it corresponds to a vector field. One can see that
 it preserves
the Poisson bracket (as  will be made clear later),
and its class in the Poisson cohomology is called the 
{\em modular class}.

This construction may be reformulated in terms of the 
 BV-structure on 
 $(\X ^\bullet(M),\Delta_\div)$, and for
  the same price, one can define a modular classe 
for any Lie Rinehart algebra together with a
divergence operator and
  differential of the form $\dpi=[\Pi,-]$, for some
   $\Pi\in\Lambda ^2_AL$ satisfying $[\Pi,\Pi]=0$.
   
In the Poisson case, it is worth noting that
\[
\{f,-\}=\dpi(f)=-[\Pi,f]
\]
and $\div(X)=-\Delta(X)$. Since $\Delta$ derives the bracket and $\Delta(f)=0$ we obtain
\[
\Delta [\Pi,f]=[\Delta(\Pi),f]
\]
We have a cannonical element attached to this situation
that computes the usual modular class in the Poisson case:
\[
X_\Delta:=\Delta(\Pi)
\]
We can observe that
\[
0=\Delta[\Pi,\Pi]=[\Delta(\Pi),\Pi]-[\Pi,\Delta(\Pi)]
=2[\Delta(\Pi),\Pi]
\]
Hence $\dpi(X_\Delta)=0$ and so, $X_\Delta$ defines a class
in $H^\bullet(\Lambda ^\bullet_AL,\dpi)$. In particular
this argument shows that vector field defined 
in the Poisson case actually gives a cohomology class.

\begin{definition}\label{modclassLieRinehart}
Given a Lie Rinehart pair $(A,L)$ with divergence operator
and corresponding BV operator $\Delta$, if 
$\Pi\in\Lambda ^2_AL$ satisfies $[\Pi,\Pi]=0$, then 
its modular class is defined as the the class of 
\[
X_\Delta:=\Delta(\Pi)\in H ^1(\Lambda ^\bullet_AL,\dpi)
\]
\end{definition}

\begin{rem}
The modular class $X_\Delta$ depends on the BV operator 
$\Delta$, or equivalently on the divergence operator $\div$. If
$a_0\in A$ and we define
\[\wt\div X:=\div X- X(a_0)
\]
then $\wt\div$ is another divergence operator, the corresponding BV operator is
\[
\wt\Delta=\Delta+[a_0,-]
\]
and the corresponding modular class is
\[
X_{\wt \Delta}=\wt\Delta(\Pi)=\Delta(\Pi)+[a_0,\Pi]
=X_\Delta+\dpi(a_0)
\]
We see that $X_{\wt\Delta}$ gives the same
class in $H^1(\Lambda ^\bullet_AL,\dpi)$ as
$X_\Delta$.
\end{rem}

The first application of Theorem \ref{main} is the following:

\begin{teo}
\label{bider}
Let $(A,L)$ be a Lie-Rinehart algebra with
divergence operator $\div$ and BV operator $\Delta$.
If $\partial$ is a Gerstenhaber differential of the form
$\dpi=[\Pi,-]$ then
\[
\Delta\dpi+\dpi \Delta=[X_\Delta,-]
\]
where, as before, $X_\Delta=\Delta(\Pi)\in L$.
In particular, if $\Pi$ is the bivector of a Poisson 
structure on a manifold $M$, then the derivation
$\Delta\dpi+\dpi\Delta$ is given by its modular class.
\end{teo}

\begin{rem}
It is well-known that the Lie algebra of derivations contains
the set of inner derivations as an ideal. However,
 in the above case, it is not true that 
$\Delta$ is a derivation, but rather a differential operator
of degree 2. Nevertheless, the formula 
in Theorem \ref{bider}       works as if $\Delta$
 were a derivation.
\end{rem}

\begin{proof}
From Theorem \ref{main} we know
$\Delta\dpi+\dpi \Delta$ is a derivation of
$\Lambda_A ^\bullet L$. In order
to show that it agrees with $[X_\Delta,-]$, it is enough
to prove that it agrees on $A$ and $L$.

By abuse of notation, let us omit the sum and write
$\Pi=X_i\wedge Y_i$ instead of $\Pi=\sum_iX_i\wedge Y_i$.
If $f\in A$ then 
\[
\Delta([\Pi,f])
+[\Pi,\Delta(f)]=
\Delta([\Pi,f])
=\Delta([f,X_i\wedge Y_i])
\]
\[=\Delta(
X_i(f) Y_i-
Y_i(f) X_i
)
=\div(
X_i(f) Y_i-
Y_i(f) X_i
)
\]
\[
=
Y_i(X_i(f))
+X_i(f)\div( Y_i)-
X_i(Y_i(f))
-Y_i(f)\div( X_i)
\]
\[
=
-[X_i,Y_i](f)
+X_i(f)\div( Y_i)
-Y_i(f)\div( X_i)
\]
\[
=
\Big(-[X_i,Y_i]
+\div( Y_i)X_i
-\div( X_i)Y_i\Big)(f)
=
\Delta(\Pi)(f)
\]
For $X\in L$ let us compute first
\[
[\Delta(\Pi),X]
=
\Big[
\Delta(X_i\wedge Y_i)
,X\Big]=
\]
\[
=
\Big[
-[X_i,Y_i]-\div(X_i)Y_i+\div(Y_i)X_i
~,~X\Big]
\]
\[
=
-[[X_i,Y_i],X]-[\div(X_i)Y_i,X]+[\div(Y_i)X_i,X]
\]
using Jacobi identity and $[fY,X]=f[Y,X]-X(f)Y$
\[
(*)=
-[[X_i,X],Y_i]
-[X_i,[Y_i,X]]
-\div(X_i)[Y_i,X]
+X(\div X_i)Y_i
+\div(Y_i)[X_i,X]
-X(\div Y_i)X_i
\]
Now we compute
\[
[\Pi,\Delta X]+\Delta[\Pi,X]
=
[X_i\wedge Y_i,-\div(X)]+\Delta([X_i,X]\wedge Y_i+X_i\wedge[Y_i,X])
\]
\[
=
X_i(\div X) Y_i
-Y_i(\div X) X_i\]
\[
-[[X_i,X], Y_i]-\div([X_i,X]) Y_i+\div(Y_i)[X_i,X]
\]
\[-[X_i,[Y_i,X]]
-\div(X_i)[Y_i,X]+\div([Y_i,X])X_i
\]
Recall that the divergence operator satisfies
$
\div([Y,X])
=Y\big(\div(X)\Big)-X\big(\div(Y)\Big)
$. So, the above expresion is equal to
\[
=
X_i(\div X) Y_i
-Y_i(\div X) X_i
-[[X_i,X], Y_i]-[X_i,[Y_i,X]]
\]
\[
-X_i(\div X) Y_i
+X(\div X_i) Y_i+\div(Y_i)[X_i,X]
\]
\[
-\div(X_i)[Y_i,X]+Y_i(\div X)X_i
-X(\div Y_i)X_i
\]

\[
=
-[[X_i,X], Y_i]-[X_i,[Y_i,X]]
+X(\div X_i) Y_i+\div(Y_i)[X_i,X]
-\div(X_i)[Y_i,X]
-X(\div Y_i)X_i
\]
and this is precisely (*).

\end{proof}

\subsection{Unimodularity and BV - structure in Cohomology}

\begin{defi}
Given a Poisson manifold $M$, we say that its 
Poisson structure is {\bf unimodular} if its modular class is
 trivial. In other words if the derivation
\[
f\mapsto \div(\{f,-\})
\]
is given by a Hamiltonian vector field.
\end{defi}
In the Lie-Rinehart setting one can give the
following
natural generalization:

\begin{definition}
Given a Lie Rinehart pair $(A,L)$ with divergence operator
and corresponding BV operator $\Delta$, for an element
$\Pi\in\Lambda ^2_AL$ satisfying $[\Pi,\Pi]=0$, we say that
this structure is {\em unimodular}
if
\[
X_\Delta=\Delta(\Pi)=0\in H ^1(\Lambda ^\bullet_AL,\dpi)
\]
In other words, if there exists $a\in A$ such that
\[
\Delta(\Pi)=[\Pi,a]\in L
\]
\end{definition}

Let us momentarily introduce an auxiliary definition.

\begin{defi}
With same notation as in the previous definition, we say that
$\Pi\in\Lambda ^2_AL$ satisfying $[\Pi,\Pi]=0$
is {\bf strictly} unimodular if
\[
X_\Delta=\Delta(\Pi)=0\in L
\]
\end{defi}

As a direct corollary of Theorem \ref{bider} we obtain that for a
 strictly unimodular structure, the BV operator anti-commutes
  with the differential $\dpi$. Therefore,  in that case,
  $H ^\bullet(\Lambda ^\bullet_AL,\dpi)$ is a BV-algebra.
Fortunately, the following Lemma holds true:

\begin{lem}\label{lemaa}
Let $(A,L)$ be  a Lie Rinehart algebra with divergence
 operator, and $\Pi\in\Lambda ^2_AL$ verifying
  $[\Pi,\Pi]=0$. If
$\Pi$ is (not necesarily strictly) unimodular, that is if
exists $a\in A$ such that
\[
X_\Delta=\dpi(a)=[\Pi,a]
\]
then
\[
\wt\div(X):=\div(X)+X(a)
\]
is another divergence operator, with 
corresponding BV operator $\wt\Delta$, and
$\Pi$ is {\bf strict} unimodular with respect to $\wt\Delta$,
i.e., it verifies
\[
X_{\wt \Delta}=0
\] 
\end{lem}
\begin{proof}
We know $\wt \div(X)=\div(X)+X(a)$ 
is another divergence operator (see Remark \ref{divtilde})
 and $\wt\Delta=\Delta-[a,-]$. We conclude
 \[
\wt\Delta(\Pi)=
\Delta(\Pi)+[a,\Pi]
=[\Pi,a]
-[a,\Pi]
=0\]
\end{proof}

\begin{rem}
Since $\Delta$ and $\wt\Delta$ differs by a derivation, 
they both generates the same 
Gerstenhaber bracket.
\end{rem}

As a corollary, we get an alternative proof of a result
in \cite{CCEY}:

\begin{teo}\label{unimod}
Let $(A,L)$ be a Lie -Rinehart algebra with divergence operator 
$\div$ together with an element
 $\Pi\in\Lambda ^2_AL$ satisfying $[\Pi,\Pi]=0$.
 If $\Pi$ is unimodular, then 
 $H ^\bullet(\Lambda ^\bullet_AL,\dpi)$ is a BV 
 algebra. In particular,
 $H_{Poiss} ^\bullet(M)$ is a BV-algebra 
 when the Poisson structure is unimodular.
\end{teo}

\begin{proof}
The operator $\Delta$ does not anticommutes with
$\dpi$, so it does not descend to cohomology.
However, with the notation as in the proof of Lemma 
\ref{lemaa},
the BV operator $\wt\Delta$ does.
\end{proof}

\subsection{Unimodularity condition: the examples in $\R^2$}

\begin{ex}\label{exR2}
 Let $M=\R ^2$, every Poisson structure is
of the form
\[
\Pi=f\partial_x\wedge\partial_y
\]
for some $f\in C ^\infty(M)$.
We consider the standar divergence operator
\[
\div (a\partial_x+b\partial_y)=a_x+b_y
\]
 So, the modular class is computed by
\[
\Delta(f\partial_x\wedge\partial_y)
=
-[f\partial_x,\partial_y]-\div(f\partial_x)\partial_y
+\div(\partial_y)f\partial_x
\]
\[
=
f_y\partial_x -f_x \partial_y
\]
Notice that it is the Hamiltonian vector field associated
 to $f$ for the
usual Poisson structure $\partial_x\wedge \partial_y$, but it 
is not necesarily a Hamiltonian vector field for
$f\partial_x\wedge \partial_y$. If $g\in C^\infty(M)$ we 
have
\[
-\{g,-\}= g_y f\partial_x-g_xf\partial_y
\]
Therefore, $f\partial_x\wedge\partial_y$ 
is unimodular in case there exists $g$ such that
\[
f\nabla g =\nabla f
\]
This condition implies $g=c+\ln|f|$ for some constant $c$,
that only make sense if
 $f$ never vanishes. This corresponds exactly the the case when
 the Poisson structure is symplectic.
 \end{ex}
 
 \begin{rem}
 The vector field 
 \[
 \Delta(f\partial_x\wedge\partial_y)=f_y\partial_x-f_x\partial_y
 \]
 is orthogonal to $\nabla f$. In this example,
 the integral curves of the modular class describe the level curves of the original $f$.
 \end{rem}

\section{Applications to Lie Bialgebras}
\subsection{Lie bialgebras and the intrinsic biderivation}

If $(\g,[-,-],\delta)$ is a Lie bialgebra, then the cannonical endomorphism given by the composition
\[
\xymatrix{
\ar@/ _2ex/[rr] _{\D:=[-,-]\circ \delta} g\ar[r] ^\delta&\Lambda ^2\g\ar[r] ^{[-,-]}&\g
}\]
is a biderivation, that is, it is simultaneously a derivation
for the bracket and a coderivation with respect to the bracket
(see \cite[Proposition 2.2 and 2.3]{FJ}), and in the case $\delta(x)=\ad_x(r)$ for some $r\in \Lambda ^2\g$ then $\D=[H_r,-]$ where
$-H_r=[-,-](r)$. This element plays a prominent role
in the classification of complex and real simple 
Lie bialgebras, for instance, 
in case $r$ is given by Belavin and Drinfel'd classification
 \cite{BD},  it is proven in
\cite{AJ} that $H_r$ is a regular element, hence its centralizer
 is a Cartan subalgebra, and this is necesarily
 {\em the} Cartan subalgebra apearing in Belavin-Drinfeld's
 classification theorem. 
 We notice that in general $[r,r]\neq 0$ (e.g. the bialgebras
 structures in Belavin-Drinfeld's classifications). 
 The 
 condition for $r$ that ensures $\delta(v)=\ad_v(r)$ is a 
 Lie cobracket is
  $[r,r]\in(\Lambda ^3\g)^\g$. The elements 
  $r\in\Lambda ^2\g$
 such that $[r,r]=0$ are called triangular Lie bialgebra
 structures, they are poorly understood and
 their classification is considered wild in some sense.
 Hence, even when $\delta(x)=\ad_x(r)$ with 
 $r\in \Lambda ^2\g$, this element $H_r$ is not,
  strictly speacking,
 the modular class (in the sense of
 Definition \ref{modclassLieRinehart}) 
 associated with the Lie bialgebroid
$ (\Lambda ^\bullet \g,\dd)$. 
Moreover if a  Lie bialgebra structure on a Lie 
algebra $\g$ is given by $\delta(x)=\ad_x(r)$  for some
 $r\in\Lambda^2\g$ then $\delta$ is called a {\em
  coboundary},
and not every cobracket is coboundary
(e.g. in the abelian Lie algebra only the zero cobracket
is coboundary). This is a remarkable difference compared
to the
Poisson situation, where a Poisson structure is defined 
by an element $\Pi\in\X^2(M)$ satisfying $[\Pi,\Pi]=0$, and
$\dpi=[\Pi,-]$.
However, one 
can easily deduce from Theorem \ref{main} the following
analogous  result for Lie bialgebras:

\begin{teo}\label{teobialg}
If $\g$ is a Lie bialgebra and  $\D=[,]\circ\delta$
then
\[
\dce\circ\dd+\dd\circ\dce=-\partial_\D
\]
where $\partial_\D:\Lambda ^\bullet\g\to\Lambda ^\bullet\g$
is the only derivation in $\Lambda ^\bullet\g$ that agrees with $\D$ on $\g$:
\[
\partial_\D(x_1\wedge\cdots \wedge x_k)=
\sum_{i=1} ^k
x_1\wedge\cdots \wedge \D(x_i)\wedge\cdots \wedge x_k
\]

\end{teo}
\begin{proof}
From Theorem \ref{main} we already know that
\[
\dce\circ\dd+\dd\circ\dce
\]
is a derivation with respect to the wedge product.
Therefore, we only need to show
 that it agrees with $-\D$ in $\g$.
Now if
 $x\in \g$, then $\dce(x)=0$ so
 \[
(\dce\circ \dd+\dd\circ\dce)(x)
=
\dce( \dd(x))+0
=
-[-,-](\delta(x))=-\D(x)
\]
\end{proof}
\begin{rem}
The above result gives an alterantive proof to the fact
 that $\D$ is a derivation with respect to the Lie
  bracket of $\g$
(since $-\partial_\D$ is a derivation for the Gerstenhaber
bracket of $\Lambda ^\bullet\g$), and that is also a
 coderivation with respect to $\delta$ (since $-\partial_\D$
commutes with $\dd$).
\end{rem}

Recall a Lie bialgebra is called {\em involutive} if 
$\D(x)=0\forall x$. In that case (and only in that case)
 we have that the 
Chevalley-Eilenberg differential anti-commutes with $\dce$,
so it is well-defined in the cohomology (with respect to $\dd$).
From Theorem \ref{teobialg} we obtain the following:

\begin{coro}
If $\g$ is an involutive Lie bialgebra then
$H ^\bullet(\Lambda\g,\dd)$ is a BV-algebra.
\end{coro}

\subsection{Minimal examples of Lie Bialgebras}

If the Lie algebra is abelian then $\D$ is trivially zero,
but the smallest non-abelian Lie algebra already 
give rise to interesting
bialgebra structures and interestings $\D$'s.

Let $\g=\aff_2(k)$  the non abelian 2-dimensional Lie algebra
with basis 
$\{h,x\}$ and bracket
\[
[h,x]=x
\]
The following is the list of all isomorphism classes of Lie bialgebra structures on the underlying Lie algebra $\g$:
\[
\begin{array}{||c|rcl|ccc||}
\hline
&&\delta&&&\D& \\
\hline
\hline
1)&\delta&\equiv& 0&&\D\equiv0& \\
\hline
2)&\delta(x)&=&0&&\D(x)=0&\D=-[x,~]\\
&\delta(h)& =
 & h\wedge x&&\D(h)=x&\\
\hline
3)&\delta(x)&=&\lambda h\wedge x&&\D(x)=\lambda x&\D=[\lambda h,~]\\
(\lambda\neq 0)&\delta(h)&=&0&&\D(h)=0&\\
\hline
\end{array}
\]
Only the cocommutative case (1) is involutive (unimodular).
Case (2) is coboundary:  
$\delta(v)=\ad_v( h\wedge x)$.
Case (3) is not coboundary:  $\delta$ is {\em not} given
by $\delta(v)=\ad_v(r)$ for any $r\in\Lambda ^2\g$.
Notice that $\D$ is diagonalizable in case (3).

%
%

\subsection{The non unimodular case}

Since $\Delta\dd+\dd\Delta$ is null-homotopic (both for 
$\Delta$ and $\dd$), we can apply Theorem \ref{teobialg}
to obtain the following corollary:

\begin{coro} \label{corodiag}
Let $\g$ be a finite dimensional Lie bialgebra
and $\D=[-,-]\circ\delta$ its characteristic biderivation.
If $\D$ is diagonalizable, then
the inclusion of the subcomplex of $\D$-invariants
$\big((\Lambda ^\bullet\g) ^\D,\dd|\big)$ 
into $(\Lambda ^\bullet,\dd)$ is a quasi-isomorphism.
\[
i:\big((\Lambda ^\bullet\g) ^\D,\dd|\big)
\overset{qis}{\longrightarrow}( \Lambda ^\bullet \g,\dd)
\]
\end{coro}

\begin{proof}
Since $\D$ is diagonalizable in $\g$, the derivation
  $\partial_\D$
is also diagonalizable in $\Lambda^\bullet\g$. We
can decompose the complex 
$( \Lambda ^\bullet \g,\dd)$ into a a direct sum 
of subcomplexes
indexed 
by the eigenvalues of $\dd$. Since $\partial_\D$ induces
 the zero
 map in
cohomology, but it is a multiple of the identity in each
summand, we conclude that every direct summand
 associated to a non-zero eigenvalue is acyclic, because
a non-zero multiple of the identity equals zero on
its cohomology.
\end{proof}

\begin{ex}
For $\g=\sl_2(\C)=\C x\oplus \C h\oplus \C y$ with
structure constants
\[
[h,x]=2x,\ [x,y]=h,\ [h,y]=-2y
\] 
and cobracket given by $\delta(v)=\ad_v(x\wedge y)$,
 we have $-\D=[h,-]$. The full complex $(\Lambda ^\bullet\g,\dd$) is
 \[
 0\to \C\overset{0}{\longrightarrow} \big(\C x\oplus \C h\oplus \C y\big)
\overset{\delta}{\longrightarrow}
 \big(\C x\wedge h\oplus \C h\wedge y\oplus \C y\wedge x
 \big)
\overset{\dd}{\longrightarrow}
 \C x\wedge h\wedge y\to 0
 \]
 while the invariant subcomplex 
 $\big((\Lambda ^\bullet\g) ^\D,\dd|\big)=
\big( (\Lambda ^\bullet\g) ^h,\dd|\big)$ is
 \[
 0\to \C\overset{0}{\longrightarrow}
  \C h   \overset{0}{\longrightarrow}
    \C y\wedge x  \overset{\cong}{\longrightarrow} 
 \C x\wedge h\wedge y\to 0
 \]
 We conclude $H^0=\C$ and $H^1=\C h$. In this case
 both computations are possible but the second one is easier.
 \end{ex}

\begin{ex}
For $\g=\sl(3,\C)$ we take basis $x_1=E_{12},
x_2=E_{23},x_3=E_{23},h_1=E_{11}-E_{22},
h_2=E_{22}-E_{33},y_1=x_1 ^t,y_2=x_ 2 ^t,y_3=-x_3^t$.
We also use the convention $h_3:=h_1+h_2$. In this example,
the so-called standard $r$-matrix is given by
\[
r=x_1\wedge y_1+
x_2\wedge y_2+x_3\wedge y_3\]
So
\[
H_r=[x_1, y_1]+[x_2, y_2]+[x_3, y_3]=
h_1+h_2+h_3=2(E_{11}-E_{33})
\]
We have $\g ^{H_r}=\h=\C h_1\oplus \C h_2$.
We also have 
$\dim \Lambda^2\g=28$ and $\dim \Lambda ^3\g=56$ while
$\dim (\Lambda^2\g) ^\D=6$ and
$\dim (\Lambda^2\g) ^\D=12$. In fact, the basis for
$\dim (\Lambda^2\g) ^\D$ is given by
\[
\{h_1\wedge h_2,\ 
x_1\wedge y_1,\ 
x_2\wedge y_2,\ 
x_3\wedge y_3,\ 
x_1\wedge y_2,\ 
x_2\wedge y_1\}\]
Recall the differential $D_\delta=[r,-]=[x_1\wedge y_1+
x_2\wedge y_2+x_3\wedge y_3,-]$. In order to compute $H^\bullet(\Lambda ^\bullet\g,D_\delta)$ in small dimensions
we clearly have
\[
[r,h]=0\ \ \forall h\in \h\]
and so $H^1(\g,D_\delta)=0$. And since the differential is zero
in $\h$, the computation of $H ^2(\g,D_\delta)$ can be done by
computing
\[
H ^2(\Lambda ^2\g,D_\delta)=
H ^2((\Lambda ^\bullet\g) ^{H_r},D_\delta)=
\Ker\Big([r,-]:(\Lambda^2\g)^{H_r}\to\Lambda ^3\g\Big)\]
So, need to compute
\[
[r,h_1\wedge h_2],\ 
[r,x_1\wedge y_1],\ 
[r,x_2\wedge y_2],\ 
[r,x_3\wedge y_3],\ 
[r,x_1\wedge y_2],\ 
[r,x_2\wedge y_1]
\]
Since $[h,r]=0$ we easily get $[r,h_1\wedge h_2]=0$, for the others, one compute the following brackets:

\[
[x_i\wedge y_j,x_k\wedge y_l]=
[x_i,x_k]\wedge y_j\wedge y_l
+x_i\wedge [y_j,x_k]\wedge y_l
-x_k\wedge [x_i,y_l]\wedge y_j
-x_k\wedge x_i\wedge [y_j,y_l]
\]
So in particular
\[
[x_i\wedge y_i,x_i\wedge y_i]=-2x_i\wedge h_i\wedge y_i
\]
and
\[
[x_1\wedge y_1,x_2\wedge y_2]
=x_3\wedge y_1\wedge y_2+x_2\wedge x_1\wedge y_3
=-x_1\wedge x_2\wedge y_3+x_3\wedge y_1\wedge y_2
\]
\[
[x_1\wedge y_1,x_3\wedge y_3]
=x_1\wedge x_2\wedge y_3+x_3\wedge y_2\wedge y_1
=x_1\wedge x_2\wedge y_3-x_3\wedge y_1\wedge y_2
\]
\[
[x_2\wedge y_2,x_3\wedge y_3]
=-x_2\wedge x_1\wedge y_3-x_3\wedge y_2\wedge y_1
=x_1\wedge x_2\wedge y_3+x_3\wedge y_1\wedge y_2
\]
So we get
\[
[r,x_i\wedge y_i]=-2x_i\wedge h_i\wedge y_i
\]
We also need to compute
\[
[x_i\wedge y_i,x_1\wedge y_2]
=
[x_i,x_1]\wedge y_i\wedge y_2
+x_i\wedge [y_i,x_1]\wedge y_2
-x_1\wedge [x_i,y_2]\wedge y_i
-x_1\wedge x_i\wedge [y_i,y_2]
\]
so
\[
[x_1\wedge y_1,x_1\wedge y_2]
=
-x_1\wedge h_1\wedge y_2
\]
\[
[x_2\wedge y_2,x_1\wedge y_2]
=
-x_1\wedge h_2\wedge y_2
\]
\[
[x_3\wedge y_3,x_1\wedge y_2]
=
+x_1\wedge x_2\wedge y_3
\]
and get
\[
[r,x_1\wedge x_2]=-x_1\wedge(h_1+h_2)\wedge y_2=-x_1\wedge h_3\wedge y_2
\]
And finally, using
\[
[x_i\wedge y_i,x_2\wedge y_1]=
[x_i,x_2]\wedge y_i\wedge y_1
+x_i\wedge [y_i,x_2]\wedge y_1
-x_2\wedge [x_i,y_1]\wedge y_i
-x_2\wedge x_i\wedge [y_i,y_1]
\]
we get
\[
[x_1\wedge y_1,x_2\wedge y_1]=
-x_2\wedge h_1\wedge y_1
\]
\[
[x_2\wedge y_2,x_2\wedge y_1]=
-x_2\wedge h_2\wedge y_1
\]
\[
[x_3\wedge y_3,x_2\wedge y_1]=0 \quad\quad\quad\quad\quad
\]
so
\[
[r,x_2\wedge y_1]=-x_2\wedge (h_1+h_2)\wedge y_1
=-x_2\wedge h_3\wedge y_1\]
We sumarize the computations:
\[
\begin{array}{rcc}
(\Lambda ^2\g) ^{h_3}&\overset{[r,-]}{\longrightarrow}&(\Lambda ^3\g) ^{h_3}\\
\\
h_1\wedge h_2&\mapsto& 0\\
x_1\wedge y_1&\mapsto&-2x_1\wedge h_1\wedge y_1\\
x_2\wedge y_2&\mapsto&-2x_2\wedge h_2\wedge y_2\\
x_3\wedge y_3&\mapsto&-2x_3\wedge h_3\wedge y_3\\
x_1\wedge y_2&\mapsto&-x_1\wedge h_3\wedge y_2\\
x_2\wedge y_1&\mapsto &-x_2\wedge h_3\wedge y_1\\
\end{array}
\]
We conclude that, for $\g=\sl(3,\C)$ and $\delta(x)=\ad_x(r)$,
where $r=x_1\wedge y_1+x_2\wedge y_2+x_3\wedge y_3$,
the cohomology in low degrees is given by:
\[
\begin{array}{rcl}
H^0(\Lambda ^\bullet\g,D_\delta)&=& \C\\
H^1(\Lambda ^\bullet\g,D_\delta)&=&\h=\C h_1\oplus \C h_2\\
H^2(\Lambda ^\bullet\g,D_\delta)&=&\Lambda ^2\h=\C h_1\wedge h_2\\
\end{array}
\]
\end{ex}

\section{The example 
$(x ^2+y^2)\partial_x\wedge \partial_y$}

We conclude with an application involving  a computation
of a classical example in Poisson geometry.

There are several papers that compute the
 Poisson cohomology of
this example, here we provide a computation showing the
power of our tools.
Consider the manifold $M=\R ^2$ equipped with
the Poisson bracket given by
the following bivector:
\[
(x^2+y^2)\partial_x\wedge \partial_y
\]
We will consider the usual divergence operator in $\R ^2$:
\[
\div\big(a(x,y)\partial_x+b(x,y)\partial_y\big)
=a_x+b_y
\]
The modular class (see Example \ref{exR2})
is
\[
X_\Delta=2(y\partial_x-x\partial_y)
\]

\begin{teo}\label{R2}
Consider the Poisson Manifold 
$(\R^2,\Pi=(x ^2+y^2)\partial_x\wedge\partial_y)$.
Denote
\[
\partial_\theta:=x\partial_y-y\partial_x=\frac12X_\Delta,
\hskip 1cm
D_r:=r\partial_r=x\partial_x+y\partial_y
\]
and
\[
\Vol=\partial_x\wedge\partial_y=\frac1{r}\partial_\theta\wedge\partial_r=\frac1{r ^2}\partial_\theta\wedge D_r\]
Notice $\partial_\theta\wedge D_r=r^2\Vol=\Pi$.
Then,  the
Poisson cohomology is given by:
\[H^0(\R^2,\Pi)=\R \]
\[H^1(\R^2,\Pi)=\R \partial_\theta\oplus \R D_r \]
\[H^2(\R^2,\Pi)=\R \Vol\oplus\R\Pi\]

\end{teo}

\begin{rem}
The computation is in agreement with \cite{N} where it is
is shown that $\dim H ^1=2=\dim H ^2$ (Theorems 
4.6 for the first equality and 5.6 for the second one). In 
that work, computations are a bit longer, and there are no
explicit generators, eventhough one could get some with some
 extra work. In our approach, computations are shorter and
  generators are clear. The Gerstenhaber algebra structure 
  is also transparent:
  \[
 \begin{array}{rccccl}
  \partial_\theta\wedge D_r&=&\Pi,\\
  {}[\partial_\theta,-]&=&0&=&[\Pi,-],
  \\
{}  [D_r,\Vol]&=&\!\!\Big[D_r,\frac{1}{r^2}\partial_\theta\wedge D_r\Big]\!\!&=&-2\Vol.
\end{array}
  \]
   Besides, a formality result 
  is clearly seen: the Poisson cohomology is isomorphic to a 
  Gerstenhaber sub-algebra of $\X^\bullet \R^2$.   
Also, the BV structure can be explicitly computed:
\[
\begin{array}{ccl}
\Delta(\partial_\theta)&=&0
\\
\Delta(D_r)&=&-\div(x\partial_x+y\partial_y)=-2
\\
\Delta(\Vol)&=&0
\\
\Delta(\Pi)&=&X_\Delta=2\partial_\theta
\end{array}
\]
\end{rem}

Our proof of Theorem \ref{R2} goes through the following lines: 
The modular class $\frac12 X_\Delta=\partial_\theta$,
is a complete vector field that integrates into an action of 
$S ^1$ by rotations. 
Since $S^1$ is compact, one can use the average trick
and prove the following:

\begin{prop}
The inclusion $(\X ^\bullet(\R^2) ^{X_\Delta},\dpi|)
\to (\X^\bullet(\R^2),\dpi)$ is a quasi-isomorphism.
\end{prop}
We need to compute $C^{\infty}(\R^2)^{\partial_\theta}$,
 $\X(\R^2) ^{\partial_\theta}$ and
 $\X ^2(\R^2) ^{\partial_\theta}$, which leads to
 a sufficiently small
  complex that allows to compute cohomology 
  by simple calculation. 
We will use the following Lemma from calculus in one variable:
\begin{lem}
Let $h\in C ^\infty(\R)$,
\begin{itemize}
\item if $h(x)=h(-x)$ for all $x\in \R$ then there exist $g\in\C ^\infty(\R)$
such that $h(x)=g(x ^2)$;
\item if $h(x)=-h(-x)$ for all $x\in \R$ then there exist $g\in\C ^\infty(\R)$
such that $h(x)=xg(x ^2)$.

\end{itemize}
\end{lem}

\begin{proof}[Proof of Theorem \ref{R2}]
First we claim that

{\em 
 $f\in C^\infty(\R^2)$ verifies $(y\partial_x-x\partial_y)(f)=0$
if and only if $f(x,y)=g(r ^2)$
for some $g\in\C ^\infty(\R)$, 
where as usual $r ^2=x^2+y^2$.}

For that, we  see that
$f$ is necessarily invariant by rotations, so
\[
f(x,y)=f\big(\sqrt{x^2+y^2},0\big)
\]
Define $h(x):=f(x,0)$, clearly $h\in C^{\infty}(\R)$. But
since $f(x,0)=f(-x,0)$, the function
$h$ is even. We conclude $h(t)=g(t ^2)$
for a smooth function $g$.

Now if a $f$ is a function of the form
 $f(x,y)=g(r^2)$, the first differential is given by
\[
\dpi(f)=[\Pi,f]=\Big[ r\partial_\theta\wedge\partial_r,g(r ^2)\Big]=
r \partial_r(g(r ^2))\partial_\theta=2r ^2g'(r^2)\]
We get that the kernel is given by the constant functions, so 
$H ^0=\R$, and the image is given by the vector fields of the form
\[
\Im(\dpi|)=\big\{r^2h(r^2)\partial_\theta: h\in C ^{\infty}(\R)
\big\}
\]
because every smooth function in $\R$ is the derivative of a smooth function.

In order to get $\X(\R ^2) ^{\partial_\theta}$, for a given 
$X\in \X(\R ^2)$ we consider its restriction to 
$\R^2\setminus (0,0)$, denoted by $X|$. Using polar coordinates,
  $X|=a(r,\theta)\partial_\theta+b(r,\theta)\partial_r$, then
 \[
 0=[\partial_\theta,X]
 =[\partial_\theta,a\partial_\theta+b\partial_r]=
 \]
 \[
 =a_\theta\partial_\theta
 +b_\theta\partial_r
 \]
and we get that 
 $a=a(r)$ and $b=b(r)$ are functions not depending on $\theta$,
they depend only on $r$:
\[
X|=a(r)\partial_\theta+b(r)\partial_r
\]
We claim that 
\[
a(r)=\wt a (r ^2)\ \hbox{ and } b(r)=r\wt b(r ^2)
\]
for some smooth functions $\wt a$ and $\wt b$.
Recall $r\partial_r=x\partial_x+y\partial_y\in\X(\R ^2)$
and $\partial_\theta=-y\partial_x+x\partial_y$, so
\[
\big(a(r)\partial_\theta+b(r)\partial_r\big)(x)=
-ya(r)+b(r)\frac{x}{r}\]
\[
\big(a(r)\partial_\theta+b(r)\partial_r\big)(y)=
xa(r)+b(r)\frac{y}{r}\]
If we write $X=f(x,y)\partial_x+g(x,y)\partial_y$ we get
\[f=-ya(r)+b(r)\frac{x}{r}
\]
\[g=xa(r)+b(r)\frac{y}{r}\]
So
\[
xf(x,y)+yg(x,y)=rb(r)\To xf(x,0)=|x|b(|x|)=-xf(-x)\]
Defining $h(x):=f(x,0)$ is a smooth odd function,
so $h(x)=x\wt b (x^2)$ for some smooth function $\wt b$. 

Similarly,
\[
-yf(x,y)+xg(x,y)=r ^2a(r)\]
So, for $x>0$ we get
\[
xg(x,0)=x^2a(x) = -xg(-x,0)
\]
which implies
\[
g(x,0)=xa(x) = -g(-x,0)
\]
Now define $h(x):=g(x,0)$ is (smooth and) odd, then
$h(x)=x\wt a(x ^2)$ for some smooth function $\wt a$.
That is
\[
g(x,0)=xa(x)=x\wt a(x^2)
\]
and we conclude  $a(r)=\wt a(r ^2)$ with
$\wt a\in C^\infty(\R)$.
 
 Notice $D_r:=r\partial_r=x\partial_x+y\partial_y$ is a 
 smooth vector field in $\R ^2$.

 We conclude that the invariant vector fields are of  the form
 \[
\X(\R ^2) ^{\partial_\theta}=
 \Big\{
 a(r ^2)\partial_\theta+b(r ^2)D_r:
 a,b\in C ^\infty(\R)\Big\}
 \]
 Let us compute the differential:
 \[
 \dpi(a(r ^2)\partial_\theta+b(r ^2)D_r)
 =
 \Big[\Pi,a(r ^2)\partial_\theta+b(r ^2)D_r\Big]
 =
 \Big[\partial_\theta\wedge D_r,a(r ^2)\partial_\theta+b(r ^2)D_r\Big]
 \]
 \[
=D_r(b(r ^2))\partial_\theta\wedge D_r 
=2r ^2b'(r ^2)\partial_\theta\wedge D_r 
\]
So, being zero means $b$ is constant, so the kernel 
of the differential is given by
\[
\Ker\dpi=\Big\{ a(r ^2)\partial_\theta+\lambda D_r:
a\in C ^\infty(\R), \ \lambda\in \R\Big\}
\]
Also, using that every smooth function in $\R$ is a derivative, we get that the image of the differential is given by
\[
\Im\dpi =
\Big\{ r ^2h(r ^2)\partial_\theta\wedge D_r :h\in C ^\infty(\R)
\}
\]

  Now we can compute the cohomology in degree one:
 \[
 H ^1(\R ^2,\dpi)=\frac{
  \Big\{
 a(r ^2)\partial_\theta+\lambda D_r:
a\in C ^\infty(\R), \ \lambda\in \R\Big\}
 }
 {
 \Big\{r^2h(r^2)\partial_\theta: h\in C ^{\infty}(\R)
\Big\} 
 }
 \]
 \[
 =
 \frac{\Big\{a(r ^2):a\in C ^\infty(\R)\Big\}}
 {\Big\{r^2h(r^2)\partial_\theta: h\in C ^{\infty}(\R)\Big\}}
 \partial_\theta \oplus \R D_r\]
 \[
=\R \partial_\theta \oplus \R D_r
 \]
 It remains to compute the invariant bivectors.
  For $\X ^2$, any bivector is a multiple of the volume form
 \[
 \partial_x\wedge\partial_y
 =\frac 1r\partial_\theta\wedge\partial_r
 =\frac 1{r ^2}\partial_\theta\wedge D_r
 \]
 So, we consider  a function $f(x,y)$ and its restriction 
 to $\R ^2\setminus (0,0)$:  $f(x,y)|=a(r ,\theta)$ and compute
 \[
 \Big[ \partial_\theta,\frac{a}r\partial_\theta\wedge\partial_r\Big]
 =
\frac{a_\theta}r\partial _\theta\wedge\partial_r
=0\]
Again from $a_\theta=0$ we get $a=a(r)=f(x,y)$ 
depends only on $r$, and the same argument used before
gives $f(x,y)=g(r ^2)$. We get
\[
(\X ^2(\R ^2)) ^{\partial_\theta}=
\Big\{
\frac{g(r ^2)}{r ^2}\partial_\theta\wedge D_r:
g\in C ^\infty(\R)
\Big\}
\]
and so, the cohomology in degree 2 is given by
\[
H ^2(\R ^2,\dpi)=\frac{
\Big\{
g(r ^2)\frac{\partial_\theta\wedge D_r}{r ^2}:
g\in C ^\infty(\R)
\Big\}
}{
\Big\{ r ^4h(r ^2)\frac{\partial_\theta\wedge D_r}{r ^2} :h\in C ^\infty(\R)
\Big\}
}
\]
\[=
\R \tfrac{\partial_\theta\wedge D_r}{r ^2}
\oplus \R r ^2\frac{\partial_\theta\wedge D_r}{r ^2}
\]
\[
=
\R \Vol\oplus \R\Pi
\]

\end{proof}


\begin{thebibliography}{999}

\bibitem[AJ]{AJ} N. Andruskiewitch, A. P. Jancsa. 
{\em On Simple Real Lie Bialgebras}. I.M.R.N. Nro. 3,
139–158 (2004).


\bibitem[BD]{BD} A. Belavin, V. Drinfeld, 
{\em Triangle Equations 
and SimpleLie Algebras}. Mathematical Physics Review, Vol.
4, Soviet Sci. Rev. Sect. C Math. Phys. Rev. 4, Harwood, 
 Chur, Switzerland, (1984), 93-165.

\bibitem[Br]{Br} J. L. Brilinsky. {\em A differential complex 
for Poisson manifolds}.
J. Differential Geom. 28(1): 93-114 (1988).


\bibitem[CYZ]{CYZ} Xiaojun Chen, Song Yang, 
Guodong Zhoub. {\em
Batalin–Vilkovisky algebras and the noncommutative Poincaré duality of Koszul Calabi–Yau algebras}
Journal of Pure and Applied Algebra
Volume 220, Issue 7, July 2016, Pages 2500-2532.


\bibitem[CCEY]{CCEY} Xiaojun Chen, Youming Chen,
 Farkhod Eshmatov and Song Yang. {\em
Poisson cohomology, Koszul duality, and Batalin-Vilkovisky
algebras},
Journal of 
Noncom. Geom. Vol. 15, No. 3 pp. 889-918 (2021).



\bibitem[FJ]{FJ} M. Farinati, A. P. Jancsa, {\em 
Trivial central extensions of Lie bialgebras},
 Journal of Algebra 390 (2013) 56-76.

\bibitem[G]{G} M.  Gerstenhaber. {\em The cohomology structure of an associative ring}. Annals of Mathematics. 78 (2): 267-288 (1963).

\bibitem[Hue]{Hue}J. Huebschmann. 
{\em 
Lie-Rinehart algebras, Gerstenhaber algebras
and Batalin-Vilkovisky algebras}.
Annales de l’institut Fourier, tome 
48, no 2 (1998), p. 425-440.



\bibitem[KS]{KS} Y. Kosmann-Schwarzbach. {\em
 Exact Gerstenhaber algebras and Lie bialgebroids},
 Acta Appl.
Math. 41 (1995), 153-165.


\bibitem[Lam]{La} T. Lambre. {\em
Dualit\'e de Van den Bergh et structure de Batalin-Vilkoviski
sur les alg\`ebres de Calabi-Yau}. 
Journal of 
Noncom. Geom. 3 (2010) 441-457.

\bibitem[Lich]{L} A. Lichnerowitch. {\em
 Les vari\'et\'es de Poisson et leurs alg\`ebres de Lie associ\'ees},
 J. Diff. Geom., 12 (1977), 253–300.


\bibitem[N]{N} N. Nakanish,
{\em Poisson Cohomology of Plane
Quadratic Poisson Structures}.
Publ. RIMS, Kyoto Univ.
33 (1997), 73-89.



\bibitem[TT]{TT} T. Tradler,
{\em The Batalin-Vilkovisky Algebra on Hochschild Cohomology Induced by Infinity Inner Products}.
Annales de l'Institut Fourier, Tome 58 (2008) no. 7, pp. 2351-2379. 

\bibitem[X]{X} Ping Xu, {\em
Gerstenhaber algebras and BV-algebras in Poisson
geometry}.
 Comm Math Phys 200, 545–560 (1999).
\end{thebibliography}
\end{document}